\documentclass[a4paper,11pt]{article}
\usepackage[utf8x]{inputenc}
\usepackage{amsmath,amsthm}
\usepackage{amsfonts,amstext,amssymb, mathtools, comment}

\usepackage{amsmath,amsthm,amssymb}
\usepackage{graphicx, subcaption}
\usepackage[a4paper,bindingoffset=0.5cm,left=2cm,right=2cm,top=2cm,bottom=2cm,footskip=.8cm]{geometry}
\usepackage[md]{titlesec}
\usepackage{color,tikz}

\newcommand{\levy}{L\'{e}vy }
\newcommand{\p}{{\mathbb P}}
\newcommand{\e}{{{\mathbb E}}}
\newcommand{\D}{\mathrm d}

\renewcommand{\a}{{\alpha}}

\renewcommand{\b}{{\beta}}

\renewcommand{\i}{{{\rm i}}}
\newcommand{\R}{{\mathbb R}}

\newcommand{\ok}{{\overline \kappa}}
\newcommand{\uk}{{\underline \kappa}}

\newtheorem{prop}{Proposition}
\newtheorem{remark}{Remark}

\begin{document}
\title{Transient analysis of a stationary L\'evy-driven queue}
\author{Jevgenijs Ivanovs$^\dagger$\, and Michel Mandjes$^{\bullet,\star}$}
\maketitle

\begin{abstract} \noindent In this paper we study a queue with \levy input, without imposing any {\it a priori} assumption on the jumps being one-sided. The focus is on computing the transforms of all sorts of quantities related to the transient workload, assuming the workload is in stationarity at time $0$.
The results are simple expressions that are in terms of the bivariate Laplace exponents of ladder processes.
 In particular, we derive the transform of the minimum workload attained over an exponentially distributed interval. 

\vspace{3mm}


\vspace{2mm}

\begin{itemize}
\item[$^\dagger$] University of Lausanne, Quartier UNIL-Dorigny
B\^{a}timent Extranef,
1015 Lausanne, Switzerland. 
\item[$^\bullet$] Korteweg-de Vries Institute for Mathematics,
University of Amsterdam, Science Park 904, 1098 XH Amsterdam, the Netherlands. 
\item[$^\star$] CWI, P.O. Box 94079, 1090 GB Amsterdam, the Netherlands.
\end{itemize}

\noindent
J.\ Ivanovs is supported by the Swiss National Science Foundation Project 200020\_143889.

\noindent 
M.\ Mandjes is also with  E{\sc urandom}, Eindhoven University of Technology, Eindhoven, the Netherlands, 
and IBIS, Faculty of Economics and Business, University of Amsterdam,
Amsterdam, the Netherlands. M. Mandjes' research is partly funded by the NWO Gravitation project NETWORKS, grant number 024.002.003.

\end{abstract}

\section{Introduction}
In this short communication we focus on the analysis of various random quantities related to the transient of a 
L{\'{e}}vy-driven queue. 
Throughout, we denote the driving  \levy process by $X_t,\,t\geq 0$, uniquely characterized through its \levy exponent 
\[\psi(\theta)={\rm Log}\,{\mathbb E} e^{{\i}\theta X_1},\quad\theta\in\R.\]
As usual, the corresponding {\it queueing process} (or: workload process) $Q_t,\,t\geq 0$ is defined as the solution to the Skorokhod problem \cite[Ch.\ II]{DM}, i.e.,  
\begin{equation}\label{eq:Q}
Q_t=Q_0+X_t-(\underline X_t+Q_0)\wedge 0=Q_0\vee (-\underline X_t)+X_t, \quad t\geq 0,
\end{equation}
where $\underline X_t:=\inf_{s\in[0,t]}X_s$ is the {\it running minimum}
and $Q_0$ is assumed to be independent of the evolution of the driving process~$X_t$ for $t\ge 0$; similarly we let $\overline X_t:=\sup_{s\in[0,t]}X_s$.

It is well known that $Q_t$ converges to a stationary distribution as $t\to\infty$ if and only if $X_t$ has a negative drift, that is, 
$\lim_{t\rightarrow\infty} X_t=-\infty$ a.s.; {\it we therefore assume that this condition is in place throughout this work}.
If $\e X_1$ is well-defined, then this condition is equivalent to requiring that $\e X_1\in [-\infty,0)$, see~\cite[Thm.\ 7.2]{KYP}.
It is also a standard result that the stationary distribution of $Q_t$ coincides with the distribution of $\overline X:=\overline X_\infty$, the {\it overall supremum}; this is a property that is often attributed to Reich \cite[Eqn.\ (2.5)]{DM}.

Importantly, in the  sequel we systematically assume that $Q_0$ {\it has the stationary distribution}; as will become clear, this assumption plays a crucial role in the analysis.
As mentioned above, the primary objective of this note is to study various transient metrics related to the process~$Q_t$.  
All our results are in terms of transforms of the quantities of interest, in addition transformed with respect to time.
In this respect, it is recalled that taking transforms with respect to time essentially amounts to considering $t=e_q$,
where $e_q$ is an exponential variable with rate $q$ (i.e., with mean $q^{-1}$), sampled independently of everything else.

This note is organized as follows. In Section 2 we sketch preliminaries, related to splitting at extrema and the Wiener-Hopf factorization. In Section 3 we focus on the distribution of the minimum workload attained over an exponentially distributed amount of time. Section 4 provides transforms of some other quantities of interest, distinguishing between two cases: (i)~the ongoing busy period and (ii)~the finished initial busy period, in which case the focus is on the so-called unused capacity. Finally, Section~5 focuses on the minimal workload in a queue conditioned to be positive. We conclude the paper by mentioning some challenging open problems.

\section{Preliminaries}\label{sec:prelim}
Let us state some well-known facts about splitting and the Wiener-Hopf factorization, all of which can be found in e.g.\ \cite[Ch.\ VI]{bertoin}.
Consider the process $X$ on the interval $[0,e_q]$ and let
\[\underline G_{e_q}:=\sup\{t\in[0,e_q]: X_t\wedge X_{t-}=\underline X_{e_q}\}\]
be the (last) time of the minimum. 
It is well known that the process $X$ {\it splits at its minimum}, see~\cite[Lem.\ VI.6]{bertoin}. This result can be stated in a simple form when considering the process $X'_t:=X_t$ for all $t\neq \underline G_{e_q}$ and $X'_{\underline G_{e_q}}:=\underline X_{e_q}$.  Namely, the processes
\[\{X'_t,t\in[0,\underline G_{e_q}]\}\:\text{ and }\:\{X'_{t+\underline G_{e_q}}-X'_{\underline G_{e_q}}, t\in[0,e_q-\underline G_{e_q}]\}\]
are independent (note that $X'$ and $X$ differ only if $X$ jumps up at $\underline G_{e_q}$ --- this jump should be included in the right process, but not in the left).
In particular, the above implies  that $\underline X_{e_q}$ and $X_{e_q}-\underline X_{e_q}$ are independent.
Splitting in the queueing context used earlier; see e.g.\  \cite{DDR}.

Let us define the following two functions for $\a,\b\in\mathbb C$ with $\Re(\a),\Re(\b)\geq 0$:
\begin{align}
\label{eq:k}
&\ok(\a,\b)=\exp\left(\int_0^\infty\int_{(0,\infty)}(e^{-t}-e^{-\a t-\b x})\frac{1}{t}\p(X_t\in \D x)\D t\right),\nonumber\\
&\uk(\a,\b)=\exp\left(\int_0^\infty\int_{(-\infty,0]}(e^{-t}-e^{-\a t-\b x})\frac{1}{t}\p(X_t\in \D x)\D t\right),
\end{align}
which are the Laplace exponents of the so-called (strictly) ascending and (weakly) descending bivariate ladder processes, respectively.
In general these functions are defined up to a multiplicative constant; the trivial scaling chosen in~\eqref{eq:k} implies that 
\begin{equation}\label{eq:kk}
\ok(\a,0)\uk(\a,0)=\exp\left(\int_0^\infty\int_{(0,\infty)}(e^{-t}-e^{-\a t})\frac{1}{t}\D t\right)=\a, \quad\a> 0,
\end{equation}
by virtue of the Frullani integral identity.
This particular scaling simplifies the calculations in the proofs, but it does not affect the results of the paper.
The Laplace exponents can then be used to express the so-called {\it Wiener-Hopf factors}:
\begin{align}\label{eq:inf}
\e e^{-\a \underline G_{e_q}+\b\underline X_{e_q}}&=\frac{\uk(q,0)}{\uk(q+\a,\b)},\\
\label{eq:sup}
\e e^{-\a (e_q-\underline G_{e_q})-\b(X_{e_q}-\underline X_{e_q})}=\e e^{-\a \overline G_{e_q}-\b\overline X_{e_q}}&=\frac{\ok(q,0)}{\ok(q+\a,\b)},
\end{align}
where $\overline G_{e_q}:=\inf\{t\in[0,e_q]: X_t\vee X_{t-}=\overline X_{e_q}\}$ is the (first) time of the maximum.

Finally, for $\b\in\mathbb R$ we have
\[\frac{q}{q-\psi(\b)}=\e e^{i\b X_{e_q}}=\e e^{i\b\underline X_{e_q}}\e e^{i\b(X_{e_q}-\underline X_{e_q})}=\frac{\uk(q,0)}{\uk(q,i\b)}\frac{\ok(q,0)}{\ok(q,-i\b)},\]
which according to~\eqref{eq:kk} yields
\begin{equation}\label{eq:WH}q-\psi(\b)=\uk(q,i\b)\ok(q,-i\b).\end{equation} 

Let us finally remark that one has to distinguish between first and last extrema 
only in the case when $X$ is a compound Poisson process; the same is true also for open and closed integration sets in~\eqref{eq:k}. 

\section{Minimum workload}
In \cite{DKM} the distribution of the minimum workload during $[0,e_q)$ was characterized for the case the driving \levy process is spectrally one-sided; it was assumed that the workload is in stationarity at time~0. In this section, this result is extended to the spectrally two-sided case. Several ramifications of this result are presented as well.

Throughout we let $\underline Q_t$ be the running minimum of the workload, i.e., $\underline Q_t$ is the minimum of $Q_s$ over $s\in[0,t]$, and we assume that $Q_0$ has the stationary workload distribution. 
\subsection{Transform of minimum workload}
This subsection shows how to express the transform of $\underline Q_t$ in terms of the Wiener-Hopf factors.
Observe that 
\begin{equation}\label{eq:Qinf}
\underline Q_t=(Q_0+\underline X_t)\vee 0=\underline X_t+Q_0\vee(-\underline X_t)
\end{equation} 
which follows by considering the event $Q_0>-\underline X_t$ and its complement separately.
To this end, we first rewrite the definition of $Q_t$, as given in~\eqref{eq:Q}, in terms of $X_t-\underline X_t$ and $\underline Q_t$:
\[Q_t=(X_t-\underline X_t)+\underline X_t+Q_0\vee(-\underline X_t)=(X_t-\underline X_t)+\underline Q_t,\]
which is also easily seen from Figure~\ref{fig:pics}.
This means that we in particular  have
\[Q_{e_q}=(X_{e_q}-\underline X_{e_q})+\underline Q_{e_q},\]
where the two terms on the right are independent, because of~\eqref{eq:Qinf} and splitting at the infimum, i.e.\ $X_{e_q}-\underline X_{e_q}$ and $\underline X_{e_q}$ are independent, see Section~\ref{sec:prelim}.
Using~\eqref{eq:sup}  we thus obtain the identity 
\begin{equation}
\label{eq:main}
\,e^{-\theta Q_{e_q}}=\e e^{-\theta (X_{e_q}-\underline X_{e_q})}\,\e e^{-\theta \underline Q_{e_q}}=
\e e^{-\theta \overline X_{e_q}}\,\e e^{-\theta \underline Q_{e_q}}=
\frac{\ok(q,0)}{\ok(q,\theta)}\,\e e^{-\theta\underline Q_{e_q}}
\end{equation} 
for $\theta\geq 0$.
Since $Q_0$ has the stationary workload distribution,   so does $Q_{e_q}.$
As a consequence, $Q_{e_q}$ has the distribution of $\overline X$, which yields 
the following result, see also~\eqref{eq:sup}.
\begin{prop}
For any $\theta,q\ge 0$,
\begin{equation}\label{eq:principal}
\e e^{-\theta\underline Q_{e_q}}=\frac{\ok(0,0)}{\ok(0,\theta)}\frac{\ok(q,\theta)}{\ok(q,0)}.
\end{equation}
\end{prop}
It is noted that Equation (\ref{eq:principal}) generalizes the findings of \cite{DKM} which just cover the spectrally one-sided cases; it is readily verified that plugging in expressions for  $\ok$ for the one-sided cases the formulas obtained are in accordance with those derived in \cite{DKM}.

\begin{remark}
The transform of $\underline Q_{e_q}$ has a simple explicit form also in the case when $Q_0$ is exponential (with mean  $\lambda^{-1}$):
\begin{align*}
\e \,e^{-\theta\underline Q_{e_q}}&=\p\left(\underline Q_{e_q}<e_{\theta}\right)=\p\left(Q_0+\underline X_{e_q}<e_{\theta}\right)=\p(Q_0<e_\theta-\underline X_{e_q})=1-\e\, e^{-\lambda(e_\theta-\underline X_{e_q})}\\
&=1-\frac{\theta}{\lambda+\theta}\frac{\uk(q,0)}{\uk(q,\lambda)}
\end{align*}
and then~\eqref{eq:main} yields $\e e^{-\theta Q_{e_q}}$.
Throughout this work we assume that $Q_0$ has the stationary distribution, but virtually all results carry over to the case of the exponential initial distribution. 
\end{remark}

We conclude this section by deriving a related, elegant identity. To this end, we define by
\[\tau:=\inf\{t\geq 0:Q_t\wedge Q_{t-}= 0\}=\inf\{t\geq 0:Q_t=0\}\text{ a.s.,}\]
the time when the system becomes empty for the first time; $\tau$ is also referred to  as the {\it residual busy period} as seen from time 0, and  the equality follows from~\cite[Prop.\ VI.4]{bertoin}.
 Clearly, $\p(\tau=e_q)=0$ and as a consequence\[\p\left(\underline Q_{e_q}=0\right)=\p(\tau<e_q).\]
Taking $\theta\rightarrow\infty$ in~\eqref{eq:principal}, and noting that $\ok(q,\theta)/\ok(0,\theta)\rightarrow 1$ which follows from~\eqref{eq:k}, we obtain
\begin{equation}\label{eq:busy}
\e \,e^{-q\tau}=\p\left(\tau<e_q\right)=\p\left(\underline Q_{e_q}=0\right)=\frac{\ok(0,0)}{\ok(q,0)}=\e \,e^{-q \overline G}
\end{equation} 
 showing that
$\tau$ and $\overline G:=\overline G_\infty$, the (first) time of the overall maxima, have the same distribution (which was concluded, using another line of argumentation, in \cite{DDR} as well).

\section{More refined quantities} In this section we analyze a few more refined quantities that are related to the minimum workload.
In this respect observe that
\[\underline G_{e_q}=\sup\left\{t\in[0,e_q]:Q_t\wedge Q_{t-}=\underline Q_{e_q}\right\}\]
is also the (last) time of the minimal workload, see Figure~\ref{fig:pics}.
 In the sequel we find it convenient to distinguish between the following two cases: $\tau>e_q$ and $\tau<e_q$,
i.e., if the residual busy period is still going on at time $e_q$, or not.
In the second case we look at the first and the last times when the workload attains the value~0,
i.e.\ the time $\tau$ when the first busy period finishes and the time $\underline G_{e_q}$ when the last busy period starts. 
\begin{figure}[h!]
\begin{subfigure}{.5\textwidth}
  \centering
  \includegraphics[width=.8\linewidth]{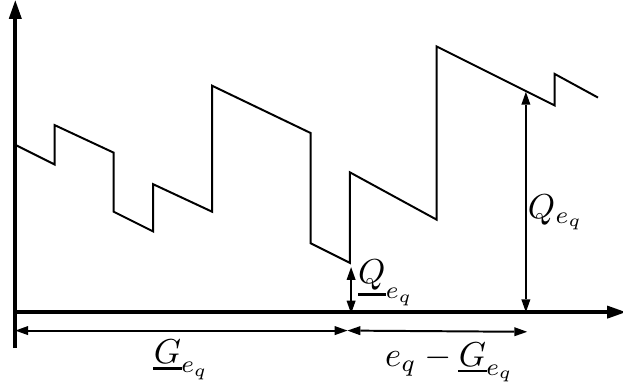}
\caption{Ongoing busy period: $\tau>e_q$}
\label{fig:ongoingBP}
\end{subfigure}%
\begin{subfigure}{0.5\textwidth}
  \centering
\includegraphics[width=.8\linewidth]{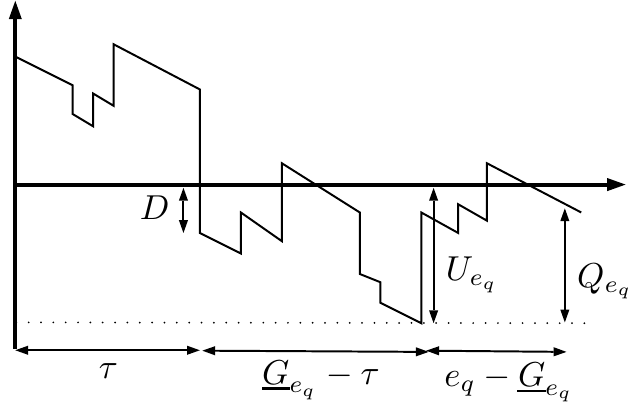}
\caption{Finished busy period: $\tau<e_q$}
\label{fig:interruptedBP}
\end{subfigure}
\caption{Schematic queueing process}
\label{fig:pics}
\end{figure}

\subsection{Ongoing busy period: $\tau>e_q$}
Note that \eqref{eq:principal} combined with~\eqref{eq:busy} provides us with the equality
\begin{eqnarray}\e \left(e^{-\theta\underline Q_{e_q}};\tau>e_q\right)&=&\e e^{-\theta\underline Q_{e_q}}-\p(\tau<e_q)\nonumber\\
&=&
\frac{\ok(0,0)}{\ok(0,\theta)}\frac{\ok(q,\theta)}{\ok(q,0)}-\frac{\ok(0,0)}{\ok(q,0)}
=\frac{\ok(0,0)(\ok(q,\theta)-\ok(0,\theta))}{\ok(0,\theta)\ok(q,0)},
\label{eq:minQ}
\end{eqnarray}
describing the distribution of $\underline Q_{e_q}$ on the event $\tau> e_q.$
The following result gives a simple expression of the joint transform of various quantities related to the minimal workload in an ongoing busy period.
\begin{prop} For any $\theta, \alpha,\beta,\gamma, q\ge 0$,
\[
\e \left(e^{-\theta\underline Q_{e_q}-\a Q_{e_q}-\b \underline G_{e_q}-\gamma(e_q-\underline G_{e_q})};\tau>e_q\right)=\frac{q}{\b+q}\frac{\ok(0,0)(\ok(q+\b,\theta+\a)-\ok(0,\theta+\a))}{\ok(0,\theta+\a)\ok(q+\gamma,\a)}.\]
\end{prop}
\begin{proof}
By appealing to the usual splitting argument, we can write
\begin{eqnarray*}
\e \left(e^{-\theta\underline Q_{e_q}-\b (e_q-\underline G_{e_q})};\tau>e_q\right)
&=&\e \left(e^{-\theta\underline Q_{e_q}};\tau>e_q\right)\e e^{-\beta(e_q-\underline G_{e_q})}\\
&=&\e \left(e^{-\theta\underline Q_{e_q}};\tau>e_q\right)\frac{\ok(q,0)}{\ok(q+\b,0)}=
\frac{\ok(0,0)(\ok(q,\theta)-\ok(0,\theta))}{\ok(0,\theta)\ok(q+\b,0)},
\end{eqnarray*}
see~\eqref{eq:sup} and~\eqref{eq:minQ}.
Now realize that
\begin{equation}\label{eq:trick}\e\left(f(e_q)e^{-\b e_q}\right)=\frac{q}{\b+q}\e f(e_{q+\b})\end{equation}
for any Borel function $f$.
It therefore immediately follows that
\[\e \left(e^{-\theta\underline Q_{e_q}-\b (e_q-\underline G_{e_q})};\tau>e_q\right)=
\frac{q}{\b+q}\e \left(e^{-\theta\underline Q_{e_{q+\b}}+\b T_{e_{q+\b}}};\tau>e_{q+\b}\right).\]
Replacing consistently $q+\beta$ by $q$,  and $\b$ by $-\b$, we thus obtain
\[\e \left(e^{-\theta\underline Q_{e_q}-\b \underline G_{e_q}};\tau>e_q\right)=
\frac{q}{\b+q}\frac{\ok(0,0)(\ok(q+\b,\theta)-\ok(0,\theta))}{\ok(0,\theta)\ok(q,0)}.\]
Finally, again relying on the splitting argument, we arrive at
\begin{eqnarray*}\lefteqn{\e \left(e^{-\theta\underline Q_{e_q}-\a Q_{e_q}-\b \underline G_{e_q}-\gamma(e_q-\underline G_{e_q})};\tau>e_q\right)}\\
&=&\e \left(e^{-(\a+\theta)\underline Q_{e_q}-\b \underline G_{e_q}};\tau>e_q\right)\e \left(e^{-\a(X_{e_q}-\underline X_{e_q})-\gamma(e_q-\underline G_{e_q})}\right)\\
&=&\frac{q}{\b+q}\frac{\ok(0,0)(\ok(q+\b,\theta+\a)-\ok(0,\theta+\a))}{\ok(0,\theta+\a)\ok(q,0)}\frac{\ok(q,0)}{\ok(q+\gamma,\a)},\end{eqnarray*}
see~\eqref{eq:sup}, completing the proof.
\end{proof}


\subsection{Finished busy period: $\tau<e_q$}
In the case $\tau<e_q$ we define two random quantities:
\begin{align*}
&U_{e_q}:=-(Q_0+\underline X_{e_q}), &D:=-(Q_0+X_\tau),
\end{align*}
see Figure~\ref{fig:pics}(b).
The first can be interpreted as the {\it unused capacity},
and the second as the sudden unused capacity at the end of the initial busy period. The latter quantity may be of limited interest, but it can be included in the joint transform without any additional work.
First we start with a basic result complementing~\eqref{eq:minQ}.
\begin{prop}\label{prop:U} For any $\theta,q\ge 0$,
\[\e\left(e^{-\theta U_{e_q}};\tau<e_q\right)=\frac{\ok(0,0)}{\ok(q,0)}\frac{\uk(q,\theta)-\uk(0,\theta)}{\uk(q,\theta)}.\]
\end{prop}
\begin{proof}
First we compute for $\theta=i\mathbb R$
\begin{align*}\e\left(e^{-\theta U_{e_q}};\tau<e_q\right)&=\e e^{\theta(Q_0+\underline X_{e_q})}-\e\left(e^{\theta \underline Q_{e_q}};\tau>e_q\right)\\
&=\frac{\ok(0,0)}{\ok(0,-\theta)}\frac{\uk(q,0)}{\uk(q,\theta)}-\frac{\ok(0,0)(\ok(q,-\theta)-\ok(0,-\theta))}{\ok(0,-\theta)\ok(q,0)},
\end{align*}
see~\eqref{eq:sup}, \eqref{eq:inf} and~\eqref{eq:minQ}. Using~\eqref{eq:kk} this can be rewritten as
\[\frac{\ok(0,0)}{\ok(q,0)\ok(0,-\theta)}\left(\frac{q}{\uk(q,\theta)}-\ok(q,-\theta)+\ok(0,-\theta)\right).\]
Finally, using~\eqref{eq:WH} we can express $\ok(q,-\theta)=(q-\psi(-i\theta))/\uk(0,\theta)$. Plugging this in yields that the expression in the previous display equals
\[
\frac{\ok(0,0)\uk(0,\theta)}{-\ok(q,0)\psi(-i\theta)}\left(\frac{\psi(-\i\theta)}{\uk(q,\theta)}-\frac{\psi(-i\theta)}{\uk(0,\theta)}\right),\]
which easily reduces to the expression in the statement.
Finally, analytic continuation shows that it is true for all $\theta\in\mathbb C$ with $\Re(\theta)\geq 0$.
\end{proof}

\begin{prop} For any $\theta, \alpha,\beta,\gamma, u,v, q\ge 0$,
\begin{eqnarray*}
\lefteqn{\e\left(e^{-\a D-\beta U_{e_q}-\gamma Q_{e_q}-u\tau -v(\underline G_{e_q}-\tau)-w(e_q-\underline G_{e_q})};\tau<e_q\right)}\\
&=&\frac{q}{q+u}\frac{\ok(0,0)(\uk(q+u,\a+\beta)-\uk(0,\a+\beta))}{\ok(q+w,\gamma)\uk(q+v,\beta)}.
\end{eqnarray*}
\end{prop}
\begin{proof}
By splitting we have 
\[\e\left(e^{-\a U_{e_q}-u (e_q-\underline G_{e_q})};\tau<e_q\right)=
\e\left(e^{-\a U_{e_q}};\tau<e_q\right)\e e^{-u (e_q-\underline G_{e_q})}=
\frac{\ok(0,0)(\uk(q,\a)-\uk(0,\a))}{\uk(q,\a)\ok(q+u,0)},\]
see also~\eqref{eq:sup} and Proposition~\ref{prop:U}.
Moreover, using \eqref{eq:trick} we write
\[\e\left(e^{-\a U_{e_q}-u (e_q-\underline G_{e_q})};\tau<e_q\right)=
\frac{q}{q+u}\e\left(e^{-\a U_{e_{q+u}}+u T_{e_{q+u}}};\tau<e_{q+u}\right).\]
As before, changing parameters leads to
\begin{eqnarray*}\e\left(e^{-\a U_{e_q}-u \underline G_{e_q}};\tau<e_q\right)&=&
\frac{q}{q+u}\e\left(e^{-\a U_{e_{q+u}}+u(e_{q+u}-T_{e_{q+u}})};\tau<e_{q+u}\right)\\
&=&
\frac{q}{q+u}\frac{\ok(0,0)(\uk(q+u,\a)-\uk(0,\a))}{\uk(q+u,\a)\ok(q,0)}.
\end{eqnarray*}
Next, we apply the strong Markov property at the time~$\tau$ to obtain the identity
\[\e\left(e^{-\a U_{e_q}-u \underline G_{e_q}};\tau<e_q\right)=\e\left(e^{-\a D-u \tau};\tau<e_q\right)\e\left(e^{\a \underline X_{e_q}-u \underline G_{e_q}}\right),\]
which leads to
\[\e\left(e^{-\a D-u \tau};\tau<e_q\right)=
\frac{\ok(0,0)(\uk(q+u,\a)-\uk(0,\a))}{q+u}\]
using~\eqref{eq:inf} and~\eqref{eq:k}.
Finally, we write using the strong Markov property at $\tau$ and splitting at $\underline G_{e_q}$
\begin{align*}
&\e\left(e^{-\a D-\beta U_{e_q}-\gamma Q_{e_q}-u\tau -v(\underline G_{e_q}-\tau)-w(e_q-\underline G_{e_q})};\tau<e_q\right)\\
&=\e\left(e^{-(\a+\b) D-u\tau};\tau<e_q\right)\e e^{\beta \underline X_{e_q}-v\underline G_{e_q}}\e e^{-\gamma (X_{e_q}-\underline X_{e_q})-w(e_q-\underline G_{e_q})}\\
&=\frac{\ok(0,0)(\uk(q+u,\a+\b)-\uk(0,\a+\b))}{q+u}\frac{\uk(q,0)}{\uk(q+v,\b)}\frac{\ok(q,0)}{\ok(q+w,\gamma)},
\end{align*}
which in view of~\eqref{eq:kk} completes the proof.
\end{proof}

\section{Minimal workload in a queue conditioned to stay positive}
In this section we focus on the law of $\underline Q_{e_q}$ conditional on the queue not having idled between $0$ and $e_q$, i.e., $\tau>e_q$, and provide an alternative representation of this law in the limiting case when 
 $q\downarrow 0$. We also comment on the relation of this limit law to \levy processes conditioned to stay positive.
 
It follows directly from~\eqref{eq:busy} and~\eqref{eq:minQ}
that 
\begin{equation}\label{eq:pos}\e \left(e^{-\theta \underline Q_{e_q}}\,|\,\tau>e_q\right)=
\left.\e\left(e^{-\theta \underline Q_{e_q}};\tau>e_q\right)\right/\p\left(\tau>e_q\right)
=\frac{\ok(0,0)}{\ok(0,\theta)}\frac{\ok(q,\theta)-\ok(0,\theta)}{\ok(q,0)-\ok(0,0)}.\end{equation}
Note, however, that the corresponding conditional law does not have a direct link (via Laplace transform) to its transient counterpart, i.e.\ when $e_q$ is replaced by~$t$.

In the following we assume that $\ok'(0,0)<\infty$ (and so also $\ok'(0,\theta)<\infty$ for $\theta\geq 0$), where the derivative of $\ok(q,\theta)$ is taken with respect to~$q$. According to~\eqref{eq:busy} this requirement is equivalent to $\e \tau=\e\overline G<\infty$; see Proposition~\ref{prop:onesided} for an example when this assumption does not hold.
Now it follows from~\eqref{eq:pos} that 
\begin{equation}\label{eq:quasy}
\lim_{q\downarrow 0}\e \left(e^{-\theta \underline Q_{e_q}}\,|\,\tau>e_q\right)=\frac{\ok(0,0)}{\ok(0,\theta)}\frac{\ok'(0,\theta)}{\ok'(0,0)}=\frac{\log(\ok(0,\theta))'}{\log(\ok(0,0))'}.
\end{equation}
Along the same lines, one can establish the generalization
\[ \lim_{q\downarrow 0}
 \e \left(e^{-\theta \underline Q_{e_q}-\alpha Q_{e_q}}\,|\,\tau>e_q\right)=\frac{\ok(0,0)}{\ok(0,\alpha)}\frac{\log(\ok(0,\theta+\alpha))'}{\log(\ok(0,0))'}.
\]

\begin{remark}
It must be noted that the limit law of $\underline Q_{e_q}|\tau>e_q$ as $q\downarrow 0$ does not coincide with the so-called quasi-stationary distribution of the minimal workload, i.e.\ with that of $\underline Q_t|\tau>t$ as $t\rightarrow\infty$, even though $e_q\rightarrow\infty$ a.s. as $q\downarrow 0$.
Roughly speaking, conditioning on $\{\tau>e_q\}$ does not only force $\tau$ to be large, but also makes $e_q$ to appear smaller. In this respect it may be helpful to mention that
\[\lim_{q\downarrow 0}\e \left(e^{-\theta e_q}\,|\,\tau>e_q\right)=\frac{1-\e\,e^{-\theta\tau}}{\theta\,\e\tau},\]
i.e.\ the limit law of $e_q|\tau>e_q$ is a proper distribution, the residual life distribution associated to~$\tau$. 
For related results on the quasi-stationary behaviour of reflected one-sided \levy processes we refer to e.g.\ \cite{MPR}.
\end{remark}
\begin{remark}\label{rem:levy_conditioned}
Importantly, one way to construct a \levy process started in $x$ and conditioned to stay positive (in the usual sense) is to condition on $\{\tau>e_q\}$ and then let $q\downarrow 0$, see~\cite[Prop.\ 1]{chaumont_doney}. The distribution of the infimum of this process is characterized in~\cite[Thm.\ 1]{chaumont_doney}.
Note, however, that the limit distribution of $\underline Q_{e_q}|\tau>e_q$ can not be obtained from this result by integrating with respect to $\p(Q_0\in \D x)$.
The main reason is that
\[\e \e_{Q_0}(e^{-\a \underline Q_{e_q}}|\tau>e_q)=\e\frac{\e_{Q_0}(e^{-\a \underline Q_{e_q}};\tau>e_q)}{\p_{Q_0}(\tau>e_q)}\neq\frac{\e \e_{Q_0}(e^{-\a \underline Q_{e_q}};\tau>e_q)}{\e \p_{Q_0}(\tau>e_q)}=\e (e^{-\a \underline Q_{e_q}}|\tau>e_q),\]
because the event upon which we condition depends on $Q_0$.
\end{remark}

\begin{prop}
Assume that $\ok'(0,0)<\infty$. Then the limit laws of 
\[\underline Q_{e_q}|\tau>e_q,\quad\text{ and }\quad X_{e_q}|X_{e_q}>0\]
coincide as $q\downarrow 0$.
\end{prop}
\begin{proof}
Using~\eqref{eq:k} observe that 
\[q\log(\ok(q,\theta))'=\int_0^\infty\int_0^\infty qe^{-qt-\theta x}\p\left(X_t\in\D x\right)\D t=\e \left(e^{-\theta X_{e_q}};X_{e_q}>0\right).\]
Hence
\[
\frac{\log(\ok(0,\theta))'}{\log(\ok(0,0))'}=\lim_{q\downarrow 0}\frac{\log(\ok(q,\theta))'}{\log(\ok(q,0))'}=\lim_{q\downarrow 0}\frac{
\e \left(e^{-\theta X_{e_q}};X_{e_q}>0\right)
}{\p\left(X_{e_q}>0\right)}\\
=\lim_{q\downarrow 0}\e \left(e^{-\theta X_{e_q}}\,|\,X_{e_q}>0\right).\]
The proof is complete in view of~\eqref{eq:quasy}.
\end{proof}
This result has a simple intuitive explanation. 
Firstly, on the left hand side we have the limit of $Q_0+\underline X_{e_q}|Q_0+\underline X_{e_q}>0$. Secondly, it holds that
$X_{e_q}=(X_{e_q}-\underline X_{e_q})+\underline X_{e_q}$, where the two terms on the right are independent and the distribution of the first converges to that of $Q_0$ as $q\downarrow 0$.

Let us conclude by giving yet another representation of the limit law of $\underline Q_{e_q}|\tau>e_q$ in the spectrally one-sided cases.
\begin{prop}\label{prop:onesided}
Assume that $X$ is either spectrally positive or spectrally negative.
Then the limiting law of $\underline Q_{e_q}|\tau>e_q$ as $q\downarrow 0$ is the residual life distribution associated to $Q_0$, i.e.
\begin{equation}\label{sp}\lim_{q\downarrow 0}\e \left(e^{-\theta \underline Q_{e_q}}\,|\,\tau>e_q\right) = \frac{1-\e\, e^{-\theta Q_0}}{\theta\,{\mathbb E}Q_0}.\end{equation}
Moreover, $\ok'(0,0)=\infty$ if and only if $X$ is a spectrally positive process with ${\rm var}(X_1)=\infty$, in which case $\e Q_0=\infty$ and $\underline Q_{e_q}|\tau>e_q$ converges to $\infty$ as $q\downarrow 0$.
\end{prop}
\begin{proof}
The identity~\eqref{sp} is easily verified using~\eqref{eq:quasy} and the explicit expressions for $\ok(\cdot,\cdot)$ in both cases, see, e.g.,~\cite[Sec.\ 6.5.2]{KYP}. 
These explicit expressions also show that in the spectrally negative case we must have $\ok'(0,0)<\infty$,
whereas in the spectrally positive case $\ok'(0,0)<\infty$ iff $\phi''(0)=\infty$, where $\phi(\theta)=\log\e e^{-\theta X_1}$. The latter is equivalent to ${\rm var}(X_1)=\infty$ and implies that~\eqref{sp} results in 0.
\end{proof}
\vspace{3mm}

It seems unlikely that~\eqref{sp} holds in general. It would be interesting to characterize all the \levy processes drifting to $-\infty$ for which~\eqref{sp} is true. 
Another challenging problem is to express the joint transform $\e e^{-
\a Q_0-\beta Q_{e_q}}$ (or, closely related, the joint transform $\e e^{-
\a Q_0-\beta \underline Q_{e_q}}$) through the functions $\ok(\cdot,\cdot)$ and $\uk(\cdot,\cdot)$.

\end{document}